\documentclass[paper=a4, fontsize=10pt]{amsart} 

\usepackage{mathalfa,amsmath,amsthm,amsfonts,amssymb}
\usepackage{pgfpages}
\usepackage{marginnote}
\usepackage{youngtab}
\usepackage{framed,lambda}
\usepackage[english]{babel}
\usepackage[linktoc=all]{hyperref}
\usepackage{cleveref,enumitem,imakeidx}
\usepackage{color,xcolor,graphicx,tikz,tikz-cd}
\usepackage{stmaryrd}
\usepackage[all]{xy}
\theoremstyle{plain}
\newtheorem{lemma}[subsection]{Lemma}
\newtheorem{theorem}[subsection]{Theorem}
\newtheorem{prop}[subsection]{Proposition} 
\newtheorem{cor}[subsection]{Corollary}

\theoremstyle{definition}

\theoremstyle{remark}
\newtheorem{example}[subsection]{Example}

\numberwithin{equation}{section} 
\numberwithin{figure}{section} 
\numberwithin{table}{section} 
\numberwithin{subsection}{section} 

\makeatletter
\let\c@equation\c@subsection
\makeatother

\makeatletter
\let\c@figure\c@subsection
\makeatother

\def\tensor{\mathop{\otimes}}
\def\mid{\,\middle\vert\,}
\newif\ifdraft
\draftfalse

\ifdraft
\pgfpagesuselayout{resize to}[a3paper]
\pagecolor[rgb]{0,0,0}
\color[rgb]{0.0,0.9,0.4}
\fi
\draftfalse

\def\borel{\ensuremath{B}}
\def\torus{\ensuremath{T}}
\def\Torus{\ensuremath{\mathcal T}}
\def\Torushat{\ensuremath{\mathcal T}}
\def\N{\ensuremath{N\left(\mathbb C[t,t^{-1}]\right)}}
\def\Borel{\ensuremath{\mathcal B}}
\def\Borelhat{\ensuremath{\mathcal B}}
\def\cartan{\ensuremath{\mathfrak h}}

\def\simple{\ensuremath{\Simple_0}}
\def\Simple{\ensuremath S}
\def\roots{\ensuremath{\overset{\mathrm o}{\Delta}}}
\def\proots{\ensuremath{\roots{}^+}}

\def\ord{\ensuremath{\operatorname{ord}}}

\def\gl{\ensuremath{G}}
\def\lgl{\ensuremath{\overset{\mathrm o}{\mathfrak{g}}}}
\def\lgL{\ensuremath{\mathfrak{g}}}
\def\gO{\ensuremath{G_0}}
\def\gL{\ensuremath{\mathfrak L(\gl)}}
\def\gLhat{\ensuremath{\mathcal G}}
\def\gOhat{\ensuremath{\gLhat_0}}
\def\para{\ensuremath{P}}
\def\Para{\ensuremath{\mathcal P}}

\def\u{\ensuremath{\mathfrak u}}
\def\e{\ensuremath{E}}
\def\Y{\ensuremath{\mathrm Y}}
\def\Y{\ensuremath{Y}}
\def\W{\ensuremath {W}}
\def\E{\ensuremath{\mathrm E}}
\def\E{\ensuremath{E}}
\def\Z{\ensuremath{\mathrm Z}}                                                  
\def\Z{\ensuremath{Z}}                                                  
\def\What{\ensuremath{\widehat\W}}
\def\coroots{\ensuremath{Q}}

\def\betac{\ensuremath{q}}
\def\w{\ensuremath{\varpi}}
\def\til{\ensuremath{\sim}}

\def\image{\ensuremath{\operatorname{Im}}}
\def\fij#1.#2.{\ensuremath{\chi(#2,#1)}}
\def\f#1.{\ensuremath{
F^i_{#1}}}  
\def\row{\ensuremath{\mathcal Row}}
\def\red{\ensuremath{\mathcal Red}}
\def\blue{\ensuremath{\mathcal Blue}}
\def\set{\ensuremath{\mathcal{S}}}
\def\Lambda{\ensuremath{d}}

\def\Ni{\ensuremath{\mathcal N}}

\def\mod{\ensuremath{\mathrm{mod}\,}}
\def\dynkAnHat{$$\xymatrix{ \overset{\alpha_1}\bullet \ar@{-}[r] &\overset{\alpha_2}\bullet \ar@{-}[r] & \cdots\ar@{-}[r] & \overset{\alpha_{n-2}}\bullet \ar@{-}[r] &\overset{\alpha_{n-1}}\bullet\ar@{-}[dll]\\ &&\overset{\alpha_0}\bullet \ar@{-}[ull] & & }$$}
\def\dynkAn{$$\xymatrix{\overset{\alpha_1} \bullet \ar@{-}[r] & \overset{\alpha_2}\bullet \ar@{-}[r] & \cdots \ar@{-}[r]&\overset{\alpha_{n-2}}\bullet \ar@{-}[r] &\overset{\alpha_{n-1}} \bullet}$$}
\title{Cotangent Bundle to the Flag variety -II}

\author{V. Lakshmibai and Rahul Singh} 

\date{\normalsize\today} 

\ifdraft\usepackage{blowup}\blowUp{paper=a3}\fi

\begin{document}

\maketitle 

\begin{abstract}
Let $P$ be a parabolic subgroup in $SL_n(\mathbb C)$.
We show that there is a ${SL_n(\mathbb C)}$-stable closed subvariety of an affine Schubert variety in an infinite dimensional partial Flag variety(associated to the Kac-Moody group ${\widehat{SL_n}(\mathbb C)}$) which is a natural compactification of the cotangent bundle to $SL_n(\mathbb C)/P$.
As a consequence, we recover the Springer resolution for any orbit closure inside the variety of nilpotent matrices.
\end{abstract}

\section{Introduction}

Earlier work of Lusztig and Strickland suggests possible connections between conormal varieties to Schubert varieties in the Grassmannian on the one hand, and affine Schubert varieties on the other.
In particular, Lusztig \cite{gl} relates certain orbit closures arising from $\widehat A_h$ to affine Schubert varieties. 
In the case $h=2$, Strickland \cite{strickland} relates such orbit closures to conormal varieties of determinantal varieties; furthermore, any determinantal variety can be canonically realized as an open subset of a Schubert variety in the Grassmannian (cf. \cite{gp2}).
\\
\\
Let \gl\ be any simple algebraic group and \gLhat\ the corresponding untwisted affine Kac-Moody group.
Let \Borel\ be a Borel subgroup of \gLhat\ and $\borel=\Borel\cap\gl$, the corresponding Borel subgroup of \gl.
The set of simple roots of \gLhat\ is $\Simple=\simple\sqcup\{\alpha_0\}$, where \simple\ is the set of simple roots of \gl. 
A simple root $\alpha$ of \gl\ is called \emph{minuscule} if the coefficient of $\alpha$ in any positive root of \gl\ (written in the simple root basis) is less than or equal to $1$.
Note that $\alpha$ is minuscule if and only if there exists an isomorphism of the underlying affine Dynkin diagram taking $\alpha$ to $\alpha_0$.
Suppose now that \para\ is the parabolic subgroup of \gl\ corresponding to some subset $\Simple_\para\subset\simple$ and \Para\ the parabolic subgroup of \gLhat\ corresponding to $\Simple_\para\subset\Simple$.
We say that \para\ is cominuscule if $\simple=\Simple_\para\sqcup\{\alpha\}$ for some minuscule root $\alpha$.
As a first step towards realizing the aforementioned relationship between affine Schubert varieties and conormal varieties to the Schubert varieties in the Grassmannian, Lakshmibai et al. (cf. \cite{vl,slof}) have constructed, for cominuscule \para, a dense embedding $\eta_\para:T^*\gl/\para\hookrightarrow X_\Para(\omega)$, where $X_\Para(\omega)$ is a particular Schubert variety in $\mathcal G/\Para$.
\\
\\
We now assume that $\gl=SL_n(\mathbb C)$.
Lakshmibai et al. \cite{crv} have constructed a similar embedding $\eta_\borel:T^*\gl/\borel\hookrightarrow X_\Borel(\kappa_0)$, where $X_\Borel(\kappa_0)$ is a particular Schubert variety in $\mathcal G/\Borel$. 
Unlike in the case of $\eta_\para$, for \para\ cominuscule, the map $\eta_\borel$ is not a dense embedding.
In this paper, we extend the result of \cite{crv} to all parabolic subgroups \para\ of \gl. 
To accomplish this, we first start with the identification of the cotangent bundle $T^*\gl/\para$ as the fiber bundle over $\gl/\para$ associated to the principal \para-bundle $\gl\rightarrow\gl/\para$, for the adjoint action of \para\ on \u($=\mathfrak{Lie}(U)$), the Lie algebra of the unipotent radical $U$ of \para.
Thus, we have the identification  $T^*\gl/\para=\gl\times^\para\u$.
We define in \Cref{defnPhi} the map $\phi_\para:T^*G/P\hookrightarrow\mathcal G/\Para$ by $\phi_\para(g,X)=g\left(1-t^{-1}X\right)(\mathrm{mod}\,\Para)$.
We then identify the minimal $\kappa\in\What^\Para$ (cf. \Cref{formula:kappa} and \Cref{injective}) for which $\image{\phi_\para}\subset X_\Para(\kappa)$.
The closure of $\image(\phi_\para)$ in $X_\Para(\kappa)$ is a \gl-stable compactification of $T^*\gl/\para$.
When \para\ is cominuscule, $\phi_\para$ is the same as the map in \cite{slof} (upto a twist by the automorphism $X\mapsto -X$ of \u).
In this case, $\phi_\para$ is a dense embedding, i.e., $X_\Para(\kappa)$ is a compactification of $T^*\gl/\para$.
\\
\\
Next we describe the connection of the above result with the Springer resolution.
Recall that the coroot lattice \coroots\ of \gl\ can be identified with the set of integer $n$-tuples $(z_1,\ldots,z_n)$ for which $\sum z_i=0$ (cf. \cite{sk}).
A coroot $\mathbf z=(z_1,\ldots,z_n)\in\coroots$ is anti-dominant if and only if $z_1\leq z_2\leq\ldots\leq z_n$.
Let \Ni\ be the variety of nilpotent $n\times n$ matrices, and \gOhat\ the parabolic subgroup corresponding to $\simple\subset\Simple$.
The variety \Ni\ has a natural algebraic stratification into \gl-orbits, indexed by partitions of $n$ (the sizes of the Jordan blocks).
Lusztig \cite{gl} has given a \gl-equivariant isomorphism $\psi$ between $\mathcal N$ and an open subset of some \gOhat-stable Schubert variety $X_{\gOhat}(\epsilon)$.
In particular, the \gl-orbits of \Ni\ must be contained in certain \gOhat-orbits of $X_{\gOhat}(\epsilon)$; 
observe that the \gOhat-orbits in $\gLhat/\gOhat$ are indexed by $\coroots^-$, the anti-dominant coroots of \gl.
Achar, Henderson, and Riche have constructed (cf. \cite{ah,ahr}) a generalization of this correspondence for all split reductive groups.
\\
\\
Now, let \para\ be the parabolic subgroup of \gl\ corresponding to omitting the simple roots $\alpha_{d_1},\ldots,\alpha_{d_{r-1}}$.
We associate to $P$ the sequence ${\boldsymbol\lambda}=(\lambda_1,\ldots,\lambda_r)$, where $\lambda_1=d_1$, $\lambda_r=n-d_{r-1}$, and $\lambda_i=d_i-d_{i-1}$ for $2\leq i\leq r-1$.
Let $\boldsymbol\nu$ be the conjugate partition of the unique partition of $n$ obtained by permuting the $\lambda_i$ so that they are in non-increasing order.
Let $\theta$ be the generalized Springer resolution $T^*G/P\rightarrow\Ni_{\boldsymbol\nu}$ (see, for example \cite{fu,he}).
In \Cref{springer2}, we show that $\psi_{\boldsymbol\nu}\circ\theta=\operatorname{pr}\circ\,\phi_\para$, where $\psi_{\boldsymbol\nu}$ is the restriction of Lusztig's map $\psi$ to $\Ni_{\boldsymbol\nu}$ and $\operatorname{pr}$ the canonical projection from $\mathcal G/\Para$ to $\mathcal G/\gOhat$.
\Cref{springer2} gives the exact relationship between the \gl-orbits of \Ni\ and the corresponding \gOhat-orbits of $X_{\gOhat}(\epsilon)$ under Lusztig's map. 
The map $\psi$ embeds $\Ni_{\boldsymbol\nu}$ as an open subvariety of the Schubert variety $X_{\gOhat}(\tau_q)$, where $q=(1-\nu_1,\ldots,1-\nu_n)\in\coroots$ and $\tau_q$ is the associated element in \What.
\\
\\
The sections are organized as follows.
In \S2, we define the embedding $\phi_P:T^*G/P\hookrightarrow\mathcal G/\Para$ and relate it to Lusztig's embedding $\psi:\mathcal N\rightarrow\mathcal G/\gOhat$ via the Springer resolution.
In \S3, we discuss some generalities on the finite and affine Weyl groups.
In \S4, we associate a tableau to a given parabolic subgroup and prove some combinatorial results on this tableau.
In \S5, we define $\kappa\in\What$ and show that $\phi_P$ identifies a \gOhat-stable subvariety of $X_\Para(\kappa)$ as a compactification of $T^*G/P$.
In \S6, we define the Schubert variety $X_{\gOhat}(\tau_q)$, which gives a compactification of $\Ni_{\boldsymbol\nu}$ and relate it to the compactification in \S5 via the Springer resolution.
In \S7, we recover the result of \cite{vl} and further show that $X_\Para(\kappa)$ is a compactification of $T^*G/P$ if and only if $P$ is a maximal parabolic subgroup.

\ifdraft\marginpar{setup.tex}\fi
\def\gln #1.{SL_n\left(\mathbb #1\right)}
\section{Lusztig's map and the Springer resolution}
In this section, we define the embedding $\phi_P:T^*G/P\hookrightarrow\mathcal G/\Para$ and relate it to Lusztig's embedding $\psi:\mathcal N\rightarrow\mathcal G/\gOhat$ via the Springer resolution.
For the generalities on Kac-Moody groups, one may refer to \cite{sk,vk}.
\subsection{The Kac-Moody Group}
Let \gl\ be the special linear group $\gln C .$, and $\lgl=\mathfrak{sl}_n({\mathbb C})$, the Lie algebra of $\gln C .$. 
The subgroup \borel\ (resp. $\borel^-$) of upper (resp. lower) triangular matrices in \gl\ is a Borel subgroup of \gl, and the subgroup \torus\ of diagonal matrices is a maximal torus in \gl. 
The root system of \gl\ with respect to $(\borel,\torus)$ has Dynkin diagram $A_{n-1}$.
We use the standard labeling of simple roots \simple\ 
\dynkAn
Let $\roots$ be the set of roots and $\proots$ the set of positive roots.
We consider the standard construction of the affine type Kac-Moody algebra \lgL, in which\begin{align*}
\lgL=\lgl\tensor\limits_{\mathbb{C}}\mathbb{C}[t,t^{-1}]\oplus\mathbb Cc\oplus\mathbb Cd
\end{align*}
and write \gLhat\ for the corresponding maximal Kac-Moody group \gLhat. 
Let $\mathfrak L(\gl)$ be the group of \gl-valued meromorphic loops, i.e., $\mathfrak L(\gl)=\gln C [[t]] [t^{-1}].$. 
The group \gLhat\ can be realized as a two dimensional central extension
$$1\longrightarrow H\longrightarrow\gLhat\overset{\mu}{\longrightarrow}\gL\longrightarrow 1$$
Let \gO\ be the group of \gl-valued holomorphic loops near $0$, i.e., $\gO=\gln C[[t]].$ and $\pi:\gO\rightarrow\gl$ the surjective map given by $t\mapsto0$. 
Then $\Borelhat:=\mu^{-1}(\pi^{-1}(\borel))$ is a Borel subgroup of \gLhat, and $\Torushat:=\mu^{-1}(\pi^{-1}(\torus))$ is a torus in $\Borelhat$. 
Let $\pi_-:\gln C[t^{-1}].\rightarrow\gl$ be the surjective map given by $t^{-1}\mapsto0$. 
The \emph{opposite Borel} $\Borel^-$ is defined as $\Borel^-=\pi_-^{-1}(\borel^-)$.
The root system of \gLhat\ with respect to $(\Borel,\Torus)$ has Dynkin diagram $\widehat {A_{n-1}}$, with simple roots $\Simple=\{\alpha_0,\ldots,\alpha_{n-1}\}$.
\dynkAnHat
Parabolic subgroups containing \borel\ (resp. \Borel) in \gl\ (resp. \gLhat) correspond to subsets of \simple\ (resp. \Simple). 
For \para\ the parabolic subgroup in \gl\ corresponding to $S_\para\subset\simple$, the parabolic subgroup \Para\ in \gLhat\ corresponding to $S_\para\subset\Simple$ is given by $\Para=\mu^{-1}(\pi^{-1}(\para))$. 
The subgroup $\gOhat:=\mu^{-1}(\gO)$ is the parabolic subgroup corresponding to the simple roots $\simple\subset\Simple$.
\ifdraft\marginpar{\color{red}setup.tex}\fi

\ifdraft\marginpar{maps.tex}\fi
\subsection{The Cotangent bundle}
\label{defCotan}
The cotangent bundle $T^*\gl/\para$ is a vector bundle over $\gl/\para$, the fiber at any point $x\in\gl$ being the cotangent space to $\gl/\para$ at $x$; the dimension of $T^*\gl/\para$ equals $2\dim\gl/\para$.
Also, $T^*\gl/\para$ is the fiber bundle over $\gl/\para$ associated to the principal \para-bundle $\gl\rightarrow\gl/\para$, for the adjoint action of \para\ on \u($=\mathfrak{Lie}(U)$), the Lie algebra of the unipotent radical $U$ of \para.
Thus \begin{align*}
    T^*\gl/\para=\gl\times^\para\u=\gl\times\u/\til
\end{align*}
where the equivalence relation \til\ is given by $(g,\Y)\til(gp,p^{-1}\Y p),\ g\in\gl,\ \Y\in\u,\ p\in\para$.
\subsection{The map $\phi_\para$:}
\label{defnPhi}
Let $\phi_\para:\gl\times^\para\u\rightarrow\gLhat/\Para$ be defined by \begin{align*}
    \phi_\para(g,Y)&=g(1-t^{-1}Y)(\mod\Para)  &g\in\gl,\ Y\in\u
\end{align*}
For $g\in\gl,\,x\in\para$ and $\Y\in\u$, we have \begin{align*}
    \phi_\para\left(gx,x^{-1}\Y x\right)  &=gx\left(1-t^{-1}x^{-1}\Y x\right)(\mod\Para)   \\
                                    &=g\left(x-t^{-1}\Y x\right)(\mod\Para)          \\
                                    &=g\left(1-t^{-1}\Y\right)(\mod\Para)            \\
                                    &=g\,\phi_\para\left(1,\Y\right)
\end{align*}
It follows that $\phi_\para$ is well-defined and \gl-equivariant.
\begin{lemma}
The map $\phi_\para$ is injective.
\end{lemma}
\begin{proof}
Suppose $\phi_\para(g,\Y)=\phi_\para(g_1,\Y_1)$, i.e., \begin{align*}
                &g\left(1-t^{-1}\Y\right)=g_1\left(1-t^{-1}\Y_1\right)(\mod\Para) \\
    \implies    &g\left(1-t^{-1}\Y\right)=g_1\left(1-t^{-1}\Y_1\right)x
\end{align*}
for some $x\in\Para$. 
Denoting $h=g_1^{-1}g$ and $\Y'=h\Y h^{-1}$, we have
 \begin{align*}
    h\left(1-t^{-1}\Y\right)&=\left(1-t^{-1}\Y_1\right)x\\
    \implies    x           &=\left(1-t^{-1}\Y_1\right)^{-1}h\left(1-t^{-1}\Y\right)\\
                            &=\left(1-t^{-1}\Y_1\right)^{-1}\left(1-t^{-1}\Y'\right)h\\
    \implies    xh^{-1}     &=\left(1+t^{-1}\Y_1+t^{-2}\Y_2+\cdots\right)\left(1-t^{-1}\Y'\right)
\end{align*}
Now since $x\in\Para,\,h\in\gl$, the left hand side is integral, i.e., does not involve negative powers of $t$.
Hence both sides must equal identity. 
It follows $x=h\in\para=\Para\bigcap\gl$ and
$                \Y_1=\Y'=h\Y h^{-1}$.
In particular, $\left(g_1,\Y_1\right)=\left(gh^{-1},h\Y h^{-1}\right)\til\left(g,\Y\right)$ as needed.
\end{proof}
\ifdraft\marginpar{maps.tex}\fi

\ifdraft\marginpar{springer.tex}\fi
\subsection{The Nilpotent Cone}
Nilpotent $n\times n$ matrices can be classified by their Jordan types, which are indexed by the set $\mathcal Par$ of partitions of $n$.
The variety \Ni\ of nilpotent elements in $\mathfrak{sl}_n$ has the decomposition $\Ni=\bigsqcup\limits_{\boldsymbol\nu\in\mathcal Par}\Ni^\circ_{\boldsymbol\nu}$, where $\Ni_{\boldsymbol\nu}^\circ$ is the set of nilpotent matrices of Jordan type $\boldsymbol\nu$.
We write $\Ni_{\boldsymbol\nu}$ for the closure (in \Ni) of $\Ni_{\boldsymbol\nu}^\circ$.
The \emph{dominance order} $\preceq$ on $\mathcal Par$ is defined by $\mu\preceq\nu$ if and only if $\sum\limits_{j\leq i}\mu_j\leq\sum\limits_{j\leq i}\nu_j$ for all $i$.
The inclusion order on $\left\{\Ni_{\boldsymbol\nu}\mid\boldsymbol\nu\in\mathcal Par\right\}$ is the same as the dominance order on $\mathcal Par$, i.e., $\Ni_{\boldsymbol\mu}\subset\Ni_{\boldsymbol\nu}$ if and only if $\mu\preceq\nu$.

\subsection{The map $\psi$:}
\label{defn:psi}
Consider the map $\psi:\mathcal{N}\rightarrow \gLhat/\gOhat$ given by\begin{align*}
\hspace{30pt}    \psi\left(N\right)&=\left(1-t^{-1}N\right)(\mod\gOhat)  &\text{for }N\in\Ni
\end{align*}
Observe that $\psi$ is \gl-equivariant:\begin{align*}
    \psi\left(gNg^{-1}\right)  &=\left(1-t^{-1}gNg^{-1}\right)(\mod\gOhat)\\
                    &=g\left(1-t^{-1}N\right)g^{-1}(\mod\gOhat)\\
                    &=g\left(1-t^{-1}N\right)(\mod\gOhat)
\end{align*}

\begin{lemma}
The map $\psi$ is injective.
\end{lemma}
\begin{proof}
Suppose $\psi\left(N\right)=\psi\left(N_1\right)$, i.e., $\left(1-t^{-1}N\right)=\left(1-t^{-1}N_1\right)(\mod\gOhat)$. 
Then $$\left(1-t^{-1}N\right)^{-1}\left(1-t^{-1}N_1\right)=\left(1+t^{-1}N+t^{-2}N^2+\cdots\right)\left(1-t^{-1}N_1\right)\in\gOhat$$
In particular, $\left(1-t^{-1}N\right)^{-1}\left(1-t^{-1}N_1\right)$ is integral, and so must equal $1$. 
It follows that $N=N_1$.
\end{proof}

\subsection{The Springer Resolution}
\label{springer}
Let $\mathfrak n$ be Lie algebra consisting of strictly upper triangular matrices.
As in \Cref{defCotan}, with $P=B$, we have an identification of $T^*G/B$ with $G\times^B\mathfrak n$. 
The Springer resolution  (cf. \cite{springer}) $\theta:T^*\gl/\borel\rightarrow\Ni$ is given by $\theta\left(g,N\right)=gNg^{-1}$ where $g\in\gl,\, N\in\mathfrak n$.
The maps $\phi_\borel$ and $\psi$ from \Cref{defnPhi,defn:psi} fit into the following commutative diagram:
\begin{center}
\vskip -20mm 
\includegraphics[scale=0.2]{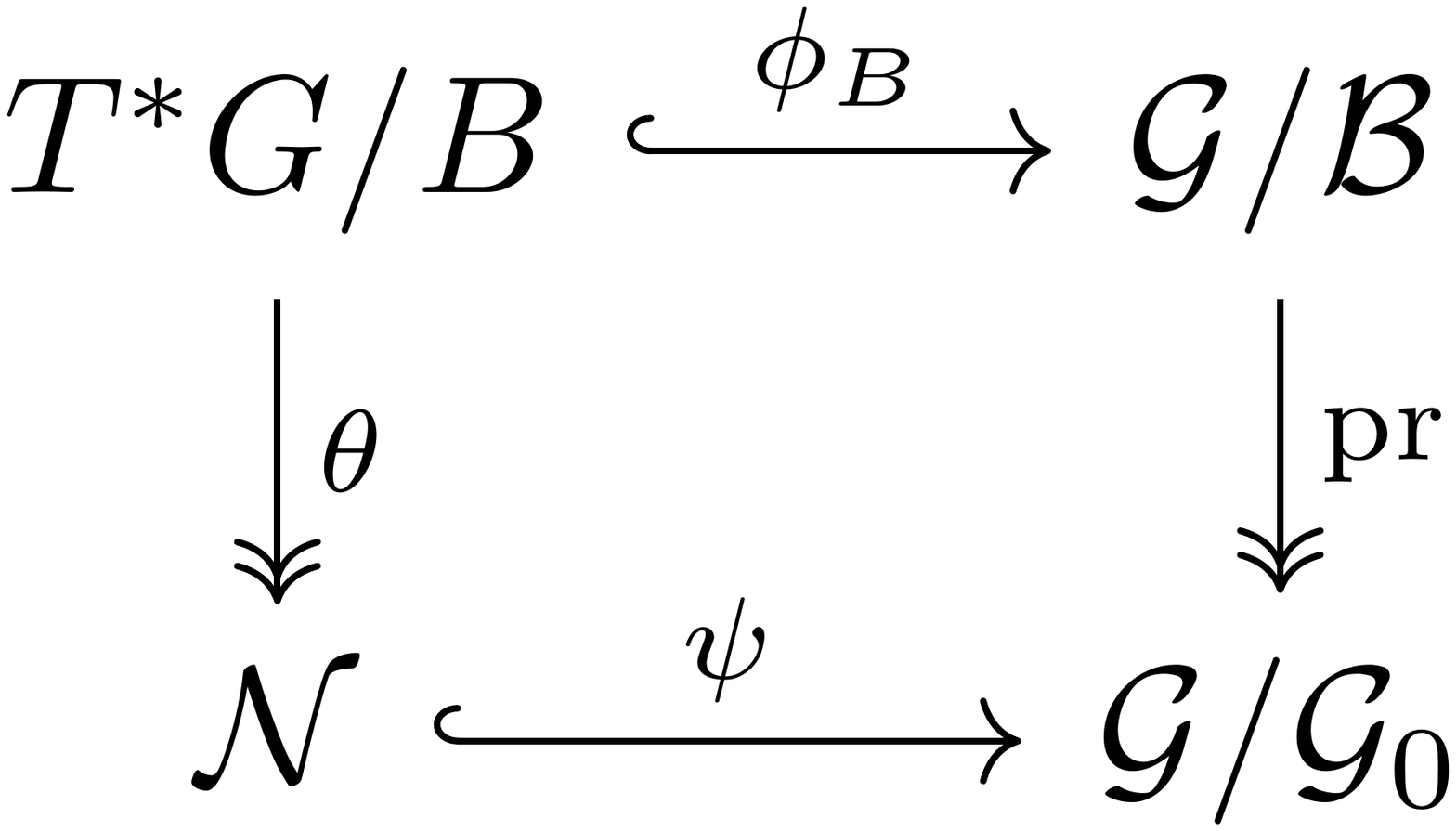}
\vspace{-20mm} 
\end{center}
where $\operatorname{pr}:\gLhat/\Borel\rightarrow\gLhat/\gOhat$ is the natural projection. 
A refinement of this result is proved in \Cref{springer2}.
\ifdraft\marginpar{springer.tex}\fi

\ifdraft\marginpar{weyl.tex}\fi
\section{The Root System and the Weyl Group}
In this section, we discuss some standard results on the root system, Weyl groups, and the Bruhat decomposition of an affine Kac-Moody group. 
For further details, the reader may refer to \cite{vk,sk}.
\subsection{The Weyl group}
\label{permPres}
Let $N$ be the normalizer of $T$ in \gl. 
The Weyl group \What\ of \gLhat\ is $N(\mathbb C[t,t^{-1}])/T$. 
The elements of $N\left(\mathbb C[t,t^{-1}]\right)$ are matrices of the form $\sum\limits_{i=1}^n t_i\E_{\sigma(i),i}$, where \begin{itemize}
    \item   $(t_i)$ is a collection of non-zero monomials in $t$ and $t^{-1}$.
    \item   $\sigma$ is a permutation of $\{1\ldots n\}$. 
    \item   $\operatorname{det}\left(\sum\limits_{i=1}^n t_i\E_{\sigma(i),i}\right)=1$.
\end{itemize}
Consider the homomorphism 
$    N\left(\mathbb C[t,t^{-1}]\right)                  \longrightarrow GL_n\left(\mathbb C[t,t^{-1}]\right) $ given by \begin{align}\label{apm}
    \sum\limits_{i=1}^n t_i\E_{\sigma(i),i} \longmapsto\sum\limits_{i=1}^n t^{\ord(t_i)}\E_{\sigma(i),i}
\end{align}
The kernel of this map is \torus, and so we can identify \What\ with the group of $n\times n$ permutation matrices $M$, with each non-zero entry a power of $t$, and $\ord(\det M)=0$.
For $w\in\What$, we call the matrix corresponding to $w$ the \emph{affine permutation matrix} of $w$.
\subsection{Generators for \What}
\label{ltau}
We shall work with the set of generators for $\What$ given by $\{s_0, s_1,\ldots, s_{n-1}\}$, where $s_i, 0\le i\le n-1$ are the reflections with respect to $\alpha_i, 0\le i\le n-1$. 
Note that $\{\alpha_i, 1\le i\le n-1\}$ being the set of simple roots of $G$, the Weyl group $W$ of $G$ is the subgroup of \What\ generated by $s_1,\ldots,s_{n-1}$.
The affine permutation matrix of $w\in W$ is $\sum\E_{w(i),i}$.
The affine permutation matrix of $s_0$ is given by
$$\begin{pmatrix}
0&0&\cdots & t^{-1}\\
0&1&\cdots &0\\
\vdots & \vdots & \vdots & \vdots\\
0&\cdots &1 &0\\
t &0&0&0
\end{pmatrix}
$$

\subsection{The root system of \gl}
Consider the vector space of $n\times n$ diagonal matrices, with basis $\left\{\E_{ii}\mid 1\leq i\leq n\right\}$.
Writing $\left\langle\epsilon_i\mid 1\leq i\leq n\right\rangle$ for the basis dual to $\left\{\E_{ii}\mid 1\leq i\leq n\right\}$, we can identify the set of roots $\roots$ of \gl\ as follows:
$$\roots=\left\{\langle\epsilon_i-\epsilon_j\rangle\mid 1\leq i\neq j\leq n\right\}$$
The positive root $\alpha_i+\ldots+\alpha_{j-1}$, where $i<j$, corresponds to $\epsilon_i-\epsilon_j$, and the negative root $-(\alpha_i+\ldots+\alpha_{j-1})$ corresponds to $\epsilon_j-\epsilon_i$.
For convenience, we shall denote the root $\epsilon_i-\epsilon_j$ by $(i,j)$, where $1\leq i\neq j\leq n$.
The action of $W(\cong S_n)$ on $\roots$ is given by $w(i,j)=(w(i),w(j))$.
For $1\leq a<b\leq n$, the reflection with respect to the positive root $(a,b)$ is given by the (affine) permutation matrix \begin{align*}
    s_{(a,b)}&=\E_{ab}+\E_{ba}+\sum\limits_{\substack{1\leq i\leq n\\ i\neq ab}} \e_{ii} 
\end{align*}

\subsection{The root system of \gLhat}
Let $\delta$ be the basic imaginary root in $\Delta$, given by $\delta=\alpha_0+\theta$, where $\theta$ is the highest root $\theta=\alpha_1+\cdots+\alpha_{n-1}$.
The root system $\Delta$ of \gLhat\ can be described as 
$$\Delta=\{n\delta+\alpha\vert n\in\mathbb Z,\alpha\in\roots\}\bigsqcup\left\{n\delta\mid n\in\mathbb Z, n\neq0\right\}$$ 
The set of positive roots $\Delta^+$ has the following description:
$$\Delta^+=\{n\delta+\alpha\vert\, n>0,\alpha\in\roots\}\bigsqcup\proots\bigsqcup\{n\delta\vert\,n>0\}$$

\subsection{Decomposition of \What\ as a semi-direct product}
\label{decomposition}
Consider the element $s_\theta\in W$, reflection with respect to the root $\theta$. 
There exists (cf. \cite{sk},\S13.1.6) a group isomorphism $\widehat{W}\rightarrow W\ltimes\coroots$ given by \begin{align*}
    s_i&\mapsto (s_i,0) &\text{ for }1\leq i\leq n-1\\
    s_0&\mapsto (s_\theta,-\theta^{\vee})&
\end{align*} 
where \coroots\ is the coroot lattice of $\mathfrak{sl}_n$ and $\theta^\vee=\alpha_1^\vee+\cdots+\alpha_{n-1}^\vee$.
The simple coroot $\alpha_i^\vee\in\cartan,\, 1\leq i\leq n-1$, is given by the matrix $\E_{i,i}-\E_{i+1,i+1}$.
As shown in \cite{crv}, a lift of $(\operatorname{id},\alpha_i^\vee)\in\What$, $1\leq i\leq n-1$ to $N(\mathbb C[t,t^{-1}])$ is given by
\begin{align*}\begin{pmatrix}
\ddots  &       &   &       \\
        & t^{-1}&   &       \\
        &       & t &       \\
        &       &   &\ddots \\
\end{pmatrix}\end{align*}
where the dots are $1$, and the off-diagonal entries are $0$.
For $q\in\coroots$, we will write $\tau_q$ for the image of $(\operatorname{id},q)\in W\ltimes\coroots$ in $\widehat W$.
In particular, if $q\in\coroots$ is given by $\tau_q=\sum t_i\e_{ii}$, and $\alpha\in\roots$ is given by $\alpha=(a,b)$, then 
\begin{align}\label{actionOfq}
    \alpha(q)=\ord(t_b)-\ord(t_a)
\end{align}
The length of $\tau_q$ is given by (cf. \cite{sk}, \S13.1.E(3)) the formula\begin{align}\label{lengthoftau}
    l(\tau_q)=\sum\limits_{\alpha\in\proots}\lvert\alpha(q)\rvert 
\end{align}

The action of $\tau_q$ on the root system of $\mathcal G$ is given by \begin{align}
\label{actionOfTau}
\hspace{30pt}    \tau_q(\alpha)&=\alpha-\alpha(q)\delta&\text{ for }\alpha\in\roots\\
\hspace{30pt}   \tau_q(\delta)&=\delta&
\end{align} 
In particular, for $\alpha\in\proots$, $\tau_q(\alpha)>0$ if and only if $\alpha(q)\leq 0$.
\begin{cor}
\label{count}
For $\alpha\in\proots$, $\tau_qs_\alpha>\tau_q$ if and only if $\alpha(q)\leq 0$ and $s_\alpha\tau_q>\tau_q$ if and only if $\alpha(q)\geq 0$ .
\end{cor}
\begin{proof}
Follows from the equivalences (cf. \cite{sk}) $ws_\alpha>w\iff w(\alpha)>0$ and $s_\alpha w>w\iff w^{-1}(\alpha)>0$ applied to $w=\tau_q$.
\end{proof}
\ifdraft\marginpar{weyl.tex}\fi

\ifdraft\marginpar{calculus.tex}\fi
\subsection{The Bruhat Decompostion}
The Bruhat decomposition of $\gLhat$ is given by $\gLhat=\bigsqcup\limits_{w\in\What}\Borelhat w\Borelhat$.
For $w\in\What$, let $X(w)\subset\gLhat/\Borelhat$ be the {\em affine Schubert variety}:
$$    X(w)=\overline{\Borelhat w\Borelhat}(\mod\Borelhat)=\bigsqcup\limits_{v\leq w}\Borelhat v\Borelhat(\mod\Borelhat)$$
The Bruhat order $\leq$ on \What\ reflects inclusion of Schubert varieties, i.e., $v\leq w$ if and only if $X(v)\subseteq X(w)$. 
The Schubert variety $X(w)$ is a projective variety of dimension $l(w)$, the length of $w$. 
\\
\\
A stronger version of the Bruhat decomposition is given by $\gLhat=\bigsqcup\limits_{w\in\What^\Para}\Borelhat w\Para$, where $\What^\Para$ is the set of minimal representatives (in the Bruhat order) of $\What/\What_\Para$ in \What.
For $w\in\What$, the {\em affine Schubert variety} $X_\Para(w)\subset\gLhat/\Para$ is given by
$$    X_\Para(w)=\overline{\Borelhat w\Para}(\mod\Para)=\bigsqcup\limits_{\substack{v\leq w\\ v,w\in\What^\Para}}\Borelhat w\Para(\mod\Para)$$
Every coset $\What_\Para w$ has a unique representative in $\What^\Para$.
For $w\in\What^\Para$, $X_\Para(w)$ is a projective variety of dimension $l(w)$.

\begin{prop}
\label{table}
Suppose $w\in\What$ is given by the affine permutation matrix 
$$\sum t_i\e_{\sigma(i)i}=\sum \e_{\sigma(i)i}\sum t_i\e_{ii}=\sigma\tau_q$$ 
with $\sigma\in S_n$ and $\tau_q=\sum t_i\e_{ii}\in\coroots$.
Let $1\leq a<b\leq n$ and set $s_r=s_{(a,b)}$, $s_l=\sigma s_r\sigma^{-1}=s_{(\sigma(a),\sigma(b))}$. 
\begin{description}
    \item[Case 1] Suppose $\ord(t_a)=\ord(t_b)$. 
        Then $s_lw=ws_r$. 
        The set $\{w,s_lw\}$ has a unique minimal element $u$, given by $u=\begin{cases} w&\text{if }\ \sigma(a)<\sigma(b)\\ s_lw&\text{if }\ \sigma(a)>\sigma(b)\end{cases}$. 
    \item[Case 2] Suppose $\ord(t_a)\neq\ord(t_b)$. 
        Then $s_lw\neq ws_r$, and the set $\{w,s_lw,ws_r,s_lws_r\}$ has a unique minimal element $u$, given by $u=\begin{cases} s_lws_r&\text{if }\ \ord(t_a)<\ord(t_b)\\ w&\text{if }\ \ord(t_a)>\ord(t_b)\end{cases}$.
        Further, we have $u<s_lu<s_lus_r$ and $u<us_r<s_lus_r$. 
\end{description}
\end{prop}
\begin{proof}
Recall from \Cref{actionOfq} that $q(a,b)=\ord(t_b)-\ord(t_a)$. 
Suppose first that $\ord(t_a)=\ord(t_b)$, which from \Cref{actionOfTau} is equivalent to $\tau_q(a,b)=(a,b)$.
It follows that 
$$ws_r=\sigma\tau_q s_{(a,b)}=\sigma s_{(a,b)}\tau_q= s_{(\sigma(a),\sigma(b)}\sigma\tau_q=s_lw$$
Recall that for $\alpha\in\proots$, we have $ws_\alpha>w$ if and only if $w(\alpha)>0$.
Now, $w(a,b)=\sigma\tau_q(a,b)=\sigma(a,b)=(\sigma(a),\sigma(b))$ being positive if and only if $\sigma(a)<\sigma(b)$, it follows that $w<ws_r$ if and only if $\sigma(a)<\sigma(b)$.
\\
\\
Suppose now that $\ord(t_a)\neq\ord(t_b)$, and set $k=\ord(t_a)-\ord(t_b)$.
Then 
$$w(a,b)=\sigma\tau_q(a,b)=(\sigma(a),\sigma(b))+k\delta$$ 
$$w^{-1}(\sigma(a),\sigma(b))=\tau_q\sigma^{-1}(\sigma(a),\sigma(b))=\tau_q(a,b)=(a,b)+k\delta$$
It follows that $w(a,b)$ and $w^{-1}(\sigma(a),\sigma(b))$ are positive if $k>0$ and negative if $k<0$.
Accordingly, we have $w<ws_r$, $w<s_lw$ if $\ord(t_a)>\ord(t_b)$, and $w>ws_r, w>s_lw$ if $\ord(t_a)<\ord(t_b)$.
Similarly, it follows from
$$s_lw(a,b)=s_l\sigma\tau_q(a,b)=s_l(\sigma(a),\sigma(b))+k\delta=(\sigma(b),\sigma(a))+k\delta$$ 
that $s_lw<s_lws_r$ if $\ord(t_a)>\ord(t_b)$ and $s_lw>s_lws_r$ if $\ord(t_a)<\ord(t_b)$.
Finally,
$$(ws_r)^{-1}(\sigma(a),\sigma(b))=s_r\tau_q^{-1}\sigma^{-1}(\sigma(a),\sigma(b))=s_r\tau_q^{-1}(a,b)=s_r(a,b)-k\delta=(b,a)-k\delta$$
implies $s_lws_r<ws_r$ if $\ord(t_a)<\ord(t_b)$ and $s_lws_r>ws_r$ if $\ord(t_a)>\ord(t_b)$.

\end{proof}
\ifdraft\marginpar{calculus.tex}\fi
\def\path{tableu}

\ifdraft\marginpar{tableu.tex}\fi
\def\e{\ensuremath{e}}
\section{Combinatorics of Parabolic Subgroups}
\label{tableau}
\label{dimofgp}
In this section, we associate a tableau to any given parabolic subgroup. 
We prove some combinatorial results on this tableau, which will be used in the succeeding sections.
We will follow the notation of this section for the rest of the paper.
\\
\\  
Let $V$ be the complex vector space $\mathbb C^n$ with standard basis $\left\{\e_i\middle\vert\ 1\leq i\leq n\right\}$.
The group \gl\ acts on $V$, and parabolic subgroups of \gl\ occur as stabilizers of partial flags in $V$. 
Consider a sequence $0=\Lambda_0<\Lambda_1<\ldots<\Lambda_{r-1}<\Lambda_r=n$.
Let $V_i$ be the subspace generated by $\left\{\e_j\middle\vert\ 1\leq j\leq d_i\right\}$, and \para\ the stabilizer of the flag $V_0\subset V_1\subset\ldots\subset V_{r-1}\subset V_r$.
Then $\para$ is the parabolic subgroup corresponding to  $\Simple_\para\subset\simple$ where $\Simple_\para=\simple\backslash\{\alpha_{\Lambda_i}|1\leq i<r\}$. 
We write \Para\ for the parabolic subgroup $\Borel\subset\Para\subset\gLhat$ corresponding to $\Simple_\para\subset\Simple$. 
Let $\lambda_i=\Lambda_i-\Lambda_{i-1}$ and $\boldsymbol\lambda=(\lambda_1,\ldots,\lambda_r)$.

\begin{figure}[h]
\centering
\begin{center}
\begin{tikzpicture}[scale=.35]\footnotesize
 \pgfmathsetmacro{\xone}{0}
 \pgfmathsetmacro{\xtwo}{ 16}
 \pgfmathsetmacro{\yone}{0}
 \pgfmathsetmacro{\ytwo}{16}

\begin{scope}<+->;
 grid
  \draw[step=1cm,gray,very thin] (\xone,\yone) grid (\xtwo,\ytwo);

\end{scope}

\begin{scope}[thick,red]
  \foreach \x in {5,8,14,15,16}
    \draw (0, 16) rectangle (5,11);
    \draw (5, 11) rectangle (8,8);
    \draw (8, 8) rectangle (10,6);
    \draw (10, 6) rectangle (11,5);
    \draw (11, 5) rectangle (14,2);
    \draw (14, 2) rectangle (16,0);
 
\draw[black] (2.5,13.5) node {\Large$\lambda_1$};
\draw[black] (6.5,9.5) node {\Large$\lambda_2$};
\draw[black] (9,7) node {\large$\bullet$};
\draw[black] (10.5,5.5) node {$\bullet$};
\draw[black] (12.5,3.5) node {\large$\bullet$};
\draw[black] (15,1) node {\large$\lambda_r$};
\end{scope}
\end{tikzpicture}
\end{center}
\caption{Elements of $\para$}
\label{diagramtofu}
\end{figure}
The elements of \para\ are exactly those matrices in \gl\ whose bottom left entries in \Cref{diagramtofu} are $0$.
Let $U$ be the unipotent radical of \para, and \u\ the Lie algebra of $U$. 
A matrix $X$ is in \u\ if and only if $X(V_i)\subset X(V_{i-1})$.
The algebra \u\ is nilpotent and contains exactly those matrices whose non-zero entries are confined to the top right corner of \Cref{diagramtofu}.
It is clear from the figure that \begin{align*}
\dim\gl/\para=\dim\u=\dfrac{n^2-\sum\lambda_i^2}{2}=\sum\limits_{i<j}\lambda_i\lambda_j
\end{align*}
\def\gln{\relax}%
\subsection{Tableaux}
\label{41}
We draw a left-aligned tableau with $r$ rows, with the $i^{th}$ row from top having $\lambda_i$ boxes. 
Fill the boxes of the tableau as follows: the entries of the $i^{th}$ row are the integers $k$ satisfying $d_i<k\leq d_{i+1}$, written in increasing order.
\\
\\
We denote by $\row(i)$ the set of entries in the $i^{th}$ row of the tableau.
The Weyl group $W_\para$ is the subgroup of $S_n$ that preserves the partition $\left\{1,\ldots,n\right\}=\bigsqcup\limits_i\row(i)$.
Let $\nu_i$ denote the number of boxes in the $i^{th}$ column from the left. 
The following are true:\begin{enumerate}
    \item The $\nu_i$ are non-increasing in $i$, i.e., $\boldsymbol\nu=(\nu_1,\ldots,\nu_s)$ is a partition of $n$.
    \item The tableau has $s:=\max\left\{\lambda_i\mid 1\leq i\leq r\right\}$ columns. 
    \item The tableau has $r:=\nu_1$ rows.
\end{enumerate}

We can now define a co-ordinate system $\fij\bullet.\bullet.$  on $\left\{1,\ldots,n\right\}$ as follows:
For $1\leq i\leq r,\ 1\leq j\leq\nu_i$, let $\fij i.j.$ denote the $j^{th}$ entry (from the top) of the $i^{th}$ column. 
Note that $\fij i.j.$ need not be in $\row(j)$.
\\
\\ 
Finally, let $\f j,k.$ denote the elementary matrix $\E_{\fij i.j.,\fij i.k.}$.
Observe that $\fij i.b.=\fij j.c.$ if and only if $i=j$ and $b=c$.
In particular, \begin{align}\label{delta}
    \f a,b.F^j_{c,d}=\delta_{ij}\delta_{bc}\f a,d.
\end{align}

\subsection{\red\ and \blue}
\label{redb4blue}
We split the set $\left\{1,\ldots,n\right\}$ into disjoint subsets \red\ and \blue, depending on their positions in the tableau. 
Let $\set_1=\left\{\fij i.1.\mid 1\leq i\leq s\right\}$ be the set of entries which are topmost in their column.
We write $\set_1(i)=\set_1\bigcap\row(i)$.
For convenience, we also define $\set_2(i)=\row(i)\backslash\set_1(i)$ and $\set_2=\bigcup\limits_i\set_2(i)$.
\\
\\
The set $\red(i)$ is the collection of the $\#\set_1(i)$ smallest entries in $\row(i)$: 
$$\red(i):=\left\{j\mid\Lambda_{i-1}<j\leq\Lambda_i-\max\left\{\lambda_j\mid j<i\right\}\right\}$$
We set $\blue(i)=\row(i)\backslash\red(i)$, $\red=\bigsqcup\limits_i\red(i)$, and $\blue=\bigsqcup\limits_i\blue(i)$.
The elements of $\red(i)$ are smaller than the elements of $\blue(i)$.
\\
\\
The elements of \red, arranged in increasing order are written $l(1),\ldots,l(s)$.
We enumerate the elements of $\blue$, written row by row from bottom to top, each row written left to right, as $m(1),\ldots,m(n-s)$.

\ifdraft\marginpar{example.tex}\fi
    \def\ten{10}
    \def\eleven{11}
    \def\eleven{11}
    \def\twelve{12}
    \def\thirteen{13}
    \def\fourteen{14}
    \def\fifteen{15}
    \def\sixteen{16}
    \def\seventeen{17}

\begin{example} Let $n=17$ and the sequence $(\Lambda_i)$ be $(1,5,9,11,17)$. 
The corresponding tableau is
\begin{align*}
    \young(1,2345,6789,\ten\eleven,\twelve\thirteen\fourteen\fifteen\sixteen\seventeen)
\end{align*}
\begin{itemize}
    \item $r=5$, $s=6$.
    \item The sequence $(\lambda_i)$ is $(1,4,4,2,6)$.
    \item The sequence $(\nu_i)$ is $(5,4,3,3,1,1)$.
    \item $\row(3)=\left\{6,7,8,9\right\}$.
    \item $\fij1.4.=10,\ \fij4.3.=15,\ \fij6.1.=17$ etc.
    \item $F^1_{2,4}=\E_{2,10},\ F^3_{3,3}=\E_{14,14}$ etc. 
    \item $\set_1=\left\{1,3,4,5,16,17\right\}$.
    \item $\red=\left\{1,2,3,4,12,13\right\}$.
    \item The sequence $m(i)$ is $(14,15,16,17,10,11,6,7,8,9,5)$.
\end{itemize}
\end{example}
\ifdraft\marginpar{example.tex}\fi

\ifdraft\marginpar{appendix.tex}\fi

\def\e{\ensuremath{e}}
\begin{prop}
Let $\boldsymbol{\nu}=(\nu_1,\ldots,\nu_r)$ and $\boldsymbol{\nu'}=(\nu'_1,\ldots,\nu'_s)$ be conjugate partitions, written in non-increasing order. 
Then \begin{align*}
    \sum\limits_{i=1}^r\nu_i^2=\sum\limits_{i=1}^s\sum\limits_{j=1}^s\min\{\nu'_i,\nu'_j\}=\sum\limits_{i=1}^s(2i-1)\nu'_i.
\end{align*}
\end{prop}
\begin{proof}
Let $\left\langle\bullet,\bullet\right\rangle$ be the dot product on $V$ given by $\langle \e_i,\e_j\rangle=\delta_{ij}$. 
Consider the Young diagram $\mathrm{\mathbf Y}$ whose $i^{th}$ row has $\nu_i$ boxes. 
Fill each box in the $i^{th}$ row of $\mathrm{\mathbf Y}$ with $\e_i$.
Let $v_i$ be the sum of all vectors in the $i^{th}$ coloumn and $\mathbf{v}$ the sum of the vectors in all the boxes in $\mathrm{\mathbf Y}$. 
Observe that $
    v_i =\sum\limits_{1\leq i\leq\nu'_i}\e_i$
and so\begin{align*}
    \langle v_i, v_j\rangle &=\left\langle\sum\limits_{k=1}^{\nu_i'}\e_k,\sum\limits_{k=1}^{\nu_j'}\e_k\right\rangle
                            =\min\{\nu'_i,\nu'_j\}\\
    \langle\mathbf v,\mathbf v\rangle   &=\sum\limits_{i=1}^s\sum\limits_{j=1}^s\langle v_i,v_j\rangle=\sum\limits_{i=1}^s\sum\limits_{j=1}^s\min\{\nu'_i,\nu'_j\}
\end{align*}
But we also have \begin{align*}
    \langle\mathbf{v},\mathbf v\rangle  &=\left\langle\sum\nu_i\e_i,\sum\nu_i\e_i\right\rangle
                                        =\sum\limits_{i,j}\delta_{ij}\nu_i\nu_j=\sum\nu_i^2 
\end{align*}
which gives the first part of the equality.
For the second part, we use $\min\{\nu'_i,\nu'_j\}=\nu'_{\max\{i,j\}}$ to get \begin{align*}
    \left\langle\mathbf v,\mathbf v\right\rangle    &=\left\langle\sum v_i,\sum\ v_j\right\rangle
                                                    =\sum\left\langle v_i, v_i\right\rangle+2\sum\limits_{i>j}\left\langle v_i, v_j\right\rangle\\
                                                    &=\sum\nu'_i+2\sum\limits_{i=1}^s\sum\limits_{j=1}^{i-1}\nu'_i 
                                                    =\sum\nu'_i+2\sum(i-1)\nu'_i
                                                    =\sum(2i-1)\nu'_i
\end{align*}
\end{proof}
\ifdraft\marginpar{appendix.tex}\fi
\begin{cor}
\label{magic}
Let $\boldsymbol\lambda, \boldsymbol\nu$ be as in \Cref{41}. Then:\begin{align*}
    \sum\limits_{i=1}^r\lambda_i^2=\sum\limits_{i=1}^s(2i-1)\nu_i.
\end{align*}\end{cor}
\begin{proof}
The collection $(\nu_1,\ldots,\nu_s)$ is a partition of $n$, and the collection $(\lambda_1,\ldots,\lambda_r)$ is a permutation of the conjugate partition $(\nu'_1,\ldots,\nu'_r)$.
\end{proof}

\ifdraft\marginpar{tableu.tex}\fi
\ifdraft\marginpar{z.tex}\fi
\section{Compactification of $T^*\gl/\para$}
\label{formula:kappa}
Let notation be as in \Cref{tableau}.
Let $\kappa\in\What$ be represented by the affine permutation matrix $$\sum\limits_{i=1}^st^{\nu_i-1}\E_{i,l(i)}+\sum\limits_{i=1}^{n-s}t^{-1}\E_{i+s,m(i)}$$
Let $\phi_\para$ be as in \Cref{defnPhi}.
In this section, we show that the image of $\phi_\para$ is contained in $X_\Para(\kappa)$. 
We first construct a matrix \Z\ whose \para-orbit is dense in \u\ and compute the Bruhat cell \w\ containing $(1-t^{-1}Z)$.
We discuss the relationship between \w\ and $\kappa$ and then prove $\image(\phi_\para)\subset X_\Para(\kappa)$.

\subsection{The Matrix \Z}
\label{defn:z}
Recall the co-ordinate system $\fij\bullet.\bullet.$ from \Cref{41}, and the elementary matrices $\f a,b.$ from \Cref{delta}.
Let \Z\ be the $n\times n$ matrix given by 
$$\Z=\sum\limits_{i=1}^{s}\sum\limits_{j=1}^{\nu_i-1}\f j,j+1.$$ 
Then 
$$Z\e_{\fij i.j.}=\begin{cases}\e_{\fij i.j-1.}&\text{if }\ j>1\\ 0&\text{if }\ j=1\end{cases}$$
It follows that $\left\{\e_{\fij i.\nu_i.}\middle\vert\ 1\leq i\leq s\right\}$ is a minimal generating set for $V$ as a module over $\mathbb C[\Z]$ (the $\mathbb C$-algebra generated by \Z).
In particular, \begin{enumerate}
    \item   $\Z(V_i)\subset V_{i-1}$, i.e., $\Z\in\u$, where \u\ is as in \Cref{defCotan}.
    \item   The Jordan type of \Z\ is $\boldsymbol\nu$. 
\end{enumerate}
An element in $x\in\u$ is called a {\em Richardson element} of \u\ if $\dim(Gx)=2\dim\gl/\para$.
Our next lemma shows that \Z\ is a Richardson element in \u.
A comprehensive study of Richardson elements in \u\ can be found in \cite{he}.

\begin{lemma}
\label{conjugacy}
\label{dense}
The \para-orbit of \Z\ is dense in \u.
Equivalently, the \gl-orbit of $(1,\Z)$ is dense in $T^*\gl/\para=\gl\times^\para\u$.
\end{lemma}
\begin{proof}
Observe that the two statements are clearly equivalent.
Let $K$ be the group of non-zero scalar matrices over $\mathbb C$, and $\widetilde\para$ the subgroup of  $\operatorname{GL}_n(\mathbb C)$ generated by $K$ and \para.
Since $K$ acts trivially on \u, it is enough to show that the orbit of \Z\ under $\widetilde\para$ is dense in \u.
Let $C_{\widetilde\para}(\Z)$ (resp. $C(\Z)$) be the stabilizer of \Z\ under the conjugation action of $\widetilde\para$ (resp. $\operatorname{GL}_n(\mathbb C)$).
It is enough to show that $\dim C_{\widetilde\para}(Z)=\dim\widetilde\para-\dim\u$.
Observe first that \begin{align*}
    \dim C(\Z)\geq\dim C_{\widetilde\para}(\Z)\geq\dim\widetilde\para-\dim\u=\sum\limits_{i=1}^r\lambda_i^2.
\end{align*}
It is now enough to show that $\dim C(\Z)\leq\sum\lambda_i^2$, from which the result will follow.
\\
\\
Let $\zeta:\operatorname{GL}_n(\mathbb C)\hookrightarrow V^{\oplus n}$ be the map $g\mapsto(ge_1,\ldots,ge_n)$, and $\rho:V^{\oplus n}\rightarrow V^{\oplus s}$ the projection $(v_1,\ldots,v_n)\mapsto(v_{\fij 1.\nu_1.},\ldots,v_{\fij s.\nu_s.})$.
As a $\mathbb C[\Z]$-module, $V$ is generated by the set $\left\{\e_{\fij i.\nu_i.}\middle\vert\ 1\leq i\leq s\right\}$. 
Since $V$ is generated (as a $\mathbb C[\Z]$-module) by the set $\left\{\e_{\fij i.\nu_i.}\middle\vert\ 1\leq i\leq s\right\}$,
the action of $g\in C(\Z)$ on $V$ is determined by its action on the set $\left\{\e_{\fij i.\nu_i.}\middle\vert\ 1\leq i\leq s\right\}$. 
In particular, the map $\left(\rho\circ\zeta\right)\vert_{C(\Z)}$ is injective.
Furthermore, since $\Z^{\nu_i}\e_{\fij i.\nu_i.}=0$, it follows that $\image(p\circ\zeta)\vert_{C(\Z)}\subset\bigoplus\limits_{i=1}^s\ker\Z^{\nu_i}$.
The kernel of $\Z^{\nu_i}$ being spanned by $\left\{e_{\fij j.k.}\middle\vert\ j\leq s,\ k\leq\nu_i,\ k\leq\nu_j\right\}$, we deduce \begin{align*}
    \dim C(\Z)          \leq&\dim\left(\bigoplus_{i=1}^s\ker\Z^{\nu_i}\right)
                        =\sum\limits_{i=1}^s\dim\left(\ker\Z^{\nu_i}\right)\\
                        &=\sum\limits_{i=1}^s\sum\limits_{j=1}^s\#\left\{\e_{\fij j.k.}\mid k\leq\min\left\{\nu_i,\nu_j\right\}\right\} \\
                        &=\sum\limits_{i=1}^s\sum\limits_{j=1}^s\min\left\{\nu_i,\nu_j\right\}
                        =\sum\limits_{i=1}^r\lambda_i^2
\end{align*}
where the last equality is from \Cref{magic}.
\end{proof}

\begin{cor}
\label{centralizerInP}
We have the equality $C_\gl(\Z)=C_\para(\Z)$. 
\end{cor}
\begin{proof}
Since $C_\gl(\Z)=C(\Z)\bigcap\gl$ and $C_\para(\Z)=C_{\widetilde\para}(\Z)\bigcap\gl$, it is enough to show $C(\Z)=C_{\widetilde\para}(\Z)$.
The map $p\circ\zeta$ embeds $C(\Z)$ as a dense, hence irreducible, subset of $V^{\oplus s}$.
In particular the real codimension of $C(\Z)$ in $V^{\oplus s}$ is at least $2$, and so $C(\Z)$ is connected. 
It follows that $C_{\widetilde\para}(\Z)$, being a codimension $0$ closed subvariety of $C(\Z)$, must equal $C(\Z)$.
\end{proof}

\ifdraft\marginpar{z.tex}\fi
\ifdraft\marginpar{main.tex}\fi

\begin{prop}
\label{kappa}
\label{wpp}
Let $\Borel\w\Borel$ be the Bruhat cell containing $1-t^{-1}\Z$. 
A lift of \w\ to \N\ is given by \begin{align*}
    \overset\sim\w=\sum\limits_{i=1}^s\left(t^{\nu_i-1}\f\nu_i,1.-\sum_{j=2}^{\nu_i}t^{-1}\f j-1,j.\right)
\end{align*}
\end{prop}
\begin{proof}
For $1\leq i\leq s$, let \begin{align*}
    b_i         &:=\sum\limits_{j=1}^{\nu_i}\sum\limits_{k=j}^{\nu_i}t^{k-j}\f k,j. 
                =\sum\limits_{j=1}^{\nu_i}\sum\limits_{k=0}^{\nu_i-j}t^k\f j+k,j. \\
    c_i         &:=\sum\limits_{j=1}^{\nu_i}\f j,j.+\sum\limits_{j=2}^{\nu_i}t^{j-1}\f j-1,1. \\
    \Z_i        &:=\sum\limits_{j=1}^{\nu_i}\f j,j.-t^{-1}\sum\limits_{j=1}^{\nu_i-1}\f j,j+1.
\end{align*}
We have\begin{align*}
    b_i\Z_ic_i  &=\left(\sum\limits_{j=1}^{\nu_i}\sum\limits_{k=0}^{\nu_i-j}t^k\f j+k,j.\right)\left(\sum\limits_{j=1}^{\nu_i}\f j,j.-t^{-1}\sum\limits_{j=1}^{\nu_i-1}\f j,j+1.\right)c_i\\
                &=\left(\sum\limits_{j=1}^{\nu_i}\sum\limits_{k=0}^{\nu_i-j}t^k\f j+k,j.-\sum_{j=1}^{\nu_i-1}\sum_{k=0}^{\nu_i-j}t^{k-1}\f j+k,j+1.\right)c_i\\
                &=\left(\sum\limits_{j=1}^{\nu_i}\sum\limits_{k=0}^{\nu_i-j}t^k\f j+k,j.-\sum_{j=2}^{\nu_i}\sum_{k=-1}^{\nu_i-j}t^k\f j+k,j.\right)c_i\\
                &=\left(\sum_{k=0}^{\nu_i-1}t^k\f 1+k,1.-\sum_{j=2}^{\nu_i}t^{-1}\f j-1,j.\right)\left(\sum\limits_{j=1}^{\nu_i}\f j,j.+\sum\limits_{j=2}^{\nu_i}t^{j-1}\f j-1,1.\right)\\
                &=\sum_{k=0}^{\nu_i-1}t^k\f k+1,1.-\sum_{j=2}^{\nu_i}t^{-1}\f j-1,j.-\sum\limits_{j=2}^{\nu_i}t^{j-2}\f j-1,1.\\
                &=t^{\nu_i-1}\f\nu_i,1.-\sum_{j=2}^{\nu_i}t^{-1}\f j-1,j.
\end{align*}
Observe that $1-t^{-1}\Z=\sum\limits_{i=1}^s\Z_i$.
It follows from \Cref{delta} that for $i\neq j$, $b_i\Z_j=0$ and $\Z_jc_i=0$.
Writing $b=\sum\limits_ib_i$ and $c=\sum\limits_ic_i$, we see \begin{align*}b(1-t^{-1}\Z)c=\sum\limits_i b_i\Z_ic_i=\w\end{align*}
The result now follows from the observation $b,c\in\Borel$.
\end{proof}

\begin{lemma}
\label{samecoset}
There exist $w_g\in\What_{\gOhat},\ w_p\in\What_\Para$ such that $\w=w_g\kappa w_p$.
In particular, $X_\Para(\w)\subset\overline{\gOhat\kappa\Para}(\mod\Para)$.
\end{lemma}
\begin{proof}
Recall the disjoint subsets $\set_1, \set_2$ of $\left\{1,\ldots,n\right\}$ from \Cref{redb4blue}.
Consider the bijection $\iota:\set_2\rightarrow\left\{\fij i.j.\mid 1\leq j\leq \nu_i-1\right\}$ given by $\iota(\fij i.j.)=\fij i.j-1.$. 
We reformulate \Cref{wpp} as 
$$    \overset\sim\w=\sum\limits_{i=1}^st^{\nu_i-1}\E_{\fij i.\nu_i.,\fij i.1.}-\sum\limits_{i\in\set_2}\E_{\iota(i),i}$$
It follows from \Cref{apm} that the affine permutation matrix of $\w$ is given by
$$    \w=\sum\limits_{i=1}^st^{\nu_i-1}\E_{\fij i.\nu_i.,\fij i.1.}+\sum\limits_{i\in\set_2}\E_{\iota(i),i}$$
Observe that $\#\red(k)=\#\set_1(k)$ and $\#\blue(k)=\#\set_2(k)$.
Since both $(l(i))_{1\leq i\leq s}$ and $(\fij i.1.)_{1\leq i\leq s}$ are increasing sequences, $\fij i.1.$ and $l(i)$ are in the same row for each $i$.
Furthermore, there exists an enumeration $t(1),\ldots,t(n-s)$ of $\set_2$ such that $t(i)$ is in the same row as $m(i)$ for all $i$.
We now define the affine permutation matrix of $w_g\in W$ and $w_p\in W_\para$ as follows: \begin{align*}
    w_g&=\sum\limits_{i=1}^s\E_{i,\fij i.\nu_i.}+\sum\limits_{i=1}^{n-s}\E_{i+s,\iota(t(i))}\\
    w_p&=\sum\limits_{i=1}^s\E_{\fij i.1.,l(i)}+\sum\limits_{i=1}^{n-s}\E_{t(i),m(i)}
\end{align*}
A simple calculation shows $\w=w_g\kappa w_p$.
\end{proof}

\begin{prop}
\label{gstable}
The Schubert variety $X_\Para(\kappa)$ is stable under left multiplication by \gOhat, i.e., $X_\Para(\kappa)=\overline{\gOhat\kappa\Para}(\mod\Para)$.
\end{prop}
\begin{proof}
Consider the affine permutation matrix of $\kappa$.
We will show, for $1\leq i<n$, either $s_i\kappa=\kappa(\mod\Para)$ or $s_i\kappa<\kappa$.
We split the proof into several cases, and use \Cref{table}:\begin{enumerate}
    \item $i<s$ and $\nu_i=\nu_{i+1}$ {\bf:} We deduce from $\nu_i=\nu_{i+1}$ that the entries $l(i)$ and $l(i+1)$ appear in the same row of the tableau. 
        In particular, $l(i)\in\Simple_\para$. 
        The non-zero entries of the $i^{th}$ and $(i+1)^{th}$ row are $t^{\nu_i-1}\E_{i,l(i)}$ and $t^{\nu_{i+1}-1}\E_{i+1,l(i+1)}$ respectively.
        We see $s_i\kappa=\kappa s_{l(i)}=\kappa(\mod\What_\Para)$. 
    \item $i<s$ and $\nu_i>\nu_{i+1}$ {\bf:} The non-zero entries of the $i^{th}$ and $(i+1)^{th}$ row are $t^{\nu_i-1}\E_{i,l(i)}$ and $t^{\nu_{i+1}-1}\E_{i+1,l(i+1)}$ respectively.
        Case $2$ of \Cref{table} applies with $a=l(i),\,b=l(i+1)$, and we have $s_i\kappa<\kappa$.
    \item $i=s$ {\bf:} The non-zero entries of the $i^{th}$ and $(i+1)^{th}$ row are $t^{\nu_s-1}\E_{s,l(s)}$ and $t^{-1}\E_{s+1,m(1)}$ respectively.
        Since $m(1)\in\blue(r)$, it follows from \Cref{redb4blue} that $l(s)<m(1)$.
        Case $2$ of \Cref{table} applied with $a=l(s),\,b=m(1)$ tells us $s_i\kappa<\kappa$.
    \item $i>s$ and $m(i-s)\in\Simple_\para$ {\bf:} The non-zero entries of the $i$ and $(i+1)^{th}$ row are $t^{-1}\E_{i,m(i-s)}$ and $t^{-1}\E_{i+1,m(i+1-s)}$ respectively.
        It follows $s_i\kappa=\kappa s_{m(i-s)}=\kappa(\mod\What_\Para)$.
    \item $i>s$ and $m(i-s)\notin\Simple_\para$ {\bf:} It follows from $m(i-s)\notin\Simple_\para$ that if $m(i-s)\in\row(j)$ then $m(i+1-s)\in\row(j-1)$.
        In particular, $m(i+1-s)<m(i-s)$.
        The non-zero entries of the $i$ and $(i+1)^{th}$ row are $t^{-1}\E_{i,m(i-s)}$ and $t^{-1}\E_{i+1,m(i+1-s)}$ respectively.
        Case $1$ of \Cref{table} applied with $a=m(i+1-s),\,b=m(i+s)$ tells us $s_i\kappa<\kappa$.
\end{enumerate}
\end{proof}

\begin{theorem}
\label{injective}
Let $\phi_\para$ be as in \Cref{defnPhi}.
Then $\image(\phi_\para)\subset X_\Para(\kappa)$.
In particular, the map $\phi_\para$ gives a compactification of $T^*\gl/\para$.
Further $\kappa$ is the minimal element in \What\ satisfying $\image(\phi_\para)\subset X_\Para(\kappa)$.
\end{theorem}
\begin{proof}
\Cref{dense} states $T^*\gl/\para=\overline{\mathcal O}$, where $\mathcal O$ is the \gl-orbit of $(1,\Z)\in\gl\times^\para\u$.
Since $\phi_\para(1,Z)\in\Borel\w\Para/\Para$ (cf. \Cref{kappa}),
the \gl-equivariance of $\phi_\para$ implies $\phi_\para(\mathcal O)\subset\gl\Borel\w\Para/\Para=\gOhat\w\Para/\Para$. 
It follows $$\phi_\para\left(T^*\gl/\para\right) =\phi_\para\left(\overline{\mathcal O}\right)\subset\overline{\phi_\para\left(\mathcal O\right)}\subset\overline{\gOhat\w\Para/\Para}=X_\Para(\kappa)$$
where the last equality follows from \Cref{samecoset} and \Cref{gstable}.
Further, $\overline{\phi_\para(T^*\gl/\para)}$, being a closed subvariety of $X_\Para(\kappa)$, is compact.
Finally, to prove minimality, we can find a $g\in\gl$ such that $\phi_\para(g,\Z)\in\Borel\kappa\Para/\Para$.
Now, \Cref{gstable} implies that $\kappa$ is maximal in $\W\w$, and so there exists a $w\in\W$ such that $\kappa=w\w$ and $l(\kappa)=l(w)+l(\w)$.
In particular, $\overline{\Borel w\Borel\w\Para}=\overline{\Borel\kappa\Para}$.
Then $\phi_\para(g,\Z)\in X_\Para(\kappa)$ for any $g\in\Borel w\Borel$.
\end{proof}

\ifdraft\marginpar{main.tex}\fi
\ifdraft\marginpar{beta.tex}\fi
\section{Lusztig's Embedding}
In this section, we construct $\betac\in\coroots$ (depending on $\boldsymbol\nu$) for which the Schubert variety $X_{\gOhat}(\tau_\betac)$ is a compactification of $\Ni_{\boldsymbol\nu}$ via $\psi$.
We also relate this compactification to the one in \Cref{injective} via the Springer resolution.
\subsection{The co-root \betac}
\label{tbetac}
\label{betac}
Let $\betac\in\coroots$ be given by the diagonal matrix:\begin{align*}
    \betac  &=\sum\limits_{i=1}^s(1-\nu_i)\E_{i,i}+\sum\limits_{i=s+1}^n\E_{i,i}\\
            &=\mathrm {Id}_{n\times n}-\sum\limits_{i=1}^s\nu_i\E_{i,i}
\end{align*}
The affine permutation matrix of $\tau_\betac$ is given by \begin{align*}
    \tau_\betac=\sum\limits_{i=1}^st^{\nu_i-1}\E_{i,i}+\sum\limits_{i=s+1}^n t^{-1}\E_{i,i}
\end{align*} 

\begin{lemma}
\label{length}
The length of $\tau_\betac(\in\What)$ is $2\dim\gl/\para$.
\end{lemma}
\begin{proof}
Let $\langle\ ,\ \rangle$ denote the dual pairing on $\cartan^*\times\cartan$.
We use \Cref{lengthoftau} to compute $l(\tau_\betac)$: \begin{align*}
    l(\tau_\betac)  &=\sum\limits_{1\leq i<j\leq n}\left\vert\left\langle \epsilon_i-\epsilon_j,\betac\right\rangle\right\vert\\
                    &=\sum\limits_{1\leq i<j\leq n}\left\lvert\left\langle \epsilon_i-\epsilon_j,\mathrm{Id}_{n\times n}-\sum\limits\nu_k\E_{kk}\right\rangle\right\rvert\\
                    &=\sum\limits_{1\leq i<j\leq n}\left\lvert\left\langle \epsilon_i-\epsilon_j,\sum\limits\nu_k\E_{kk}\right\rangle\right\rvert\\
                    &=\sum\limits_{1\leq i<j\leq s}\left\lvert\left\langle \epsilon_i-\epsilon_j,\sum\limits\nu_k\E_{kk}\right\rangle\right\rvert+\sum\limits_{\substack{1\leq i\leq s\\ s<j\leq n}}\left\lvert\left\langle\epsilon_i-\epsilon_j,\sum\limits\nu_k\E_{kk}\right\rangle\right\rvert \\ 
                    &=\sum\limits_{1\leq i<j\leq s}\left(\nu_i-\nu_j\right)+\sum\limits_{\substack{1\leq i\leq s\\ s<j\leq n}}\nu_i 
\end{align*}\begin{align*}
                    &=\sum\limits_{k=1}^{s} (s+1-2k)\nu_k+(n-s)\sum\limits_{k=1}^s\nu_i\\
                    &=s\sum\limits_{k=1}^s\nu_k - \sum\limits_{k=1}^s(2k-1)\nu_k+(n-s)n\\
                    &=sn - \sum\limits_{k=1}^{s} (2k-1)\nu_k +(n-s)n\\
                    &=n^2-\sum\lambda_i^2\qquad\qquad\qquad(\text{cf. \Cref{magic}})\\
                    &=2\dim\gl/\para.
\end{align*}
\end{proof}

\begin{lemma}
\label{betainWg}
The element $\tau_\betac$ belongs to $\What^{\gOhat}$.
In particular, $\dim X_{\gOhat}(\tau_\betac)=l(\tau_\betac)$.
\end{lemma}
\begin{proof}
Follows from \Cref{actionOfTau} and \Cref{actionOfq}.
\end{proof}

\ifdraft\marginpar{beta.tex}\fi
\ifdraft\marginpar{nilpotent.tex}\fi

\begin{theorem}
\label{springer2}
\label{lusztig}
There exists a resolution of singularities $\theta:T^*\gl/\para\rightarrow\Ni_{\boldsymbol\nu}$, given by $(g,N)\mapsto gNg^{-1}$ for $(g,N)\in\gl\times^\para\u$.
Furthermore, we have the following commutative diagram:\begin{center}
\vskip -3mm 
\includegraphics[scale=0.9]{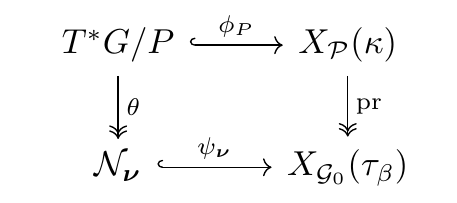}
\vspace{-2mm} 
\end{center}
where $\operatorname{pr}$ is the restriction to $X_\Para(\kappa)$ of the natural projection $\gLhat/\Para\twoheadrightarrow\gLhat/\gOhat$.
The map $\psi_\nu$ is an open immersion onto the \emph{opposite cell} $Y_{\gOhat}(\tau_q):=\Borel^-\cap X_{\gOhat}(\tau_q)$ of the Schubert variety $X_{\gOhat}(\tau_\betac)$;
\ifdraft\marginpar{The map $\psi_\nu$ is an open embedding.}\fi 
in particular, the Schubert variety $X_{\gOhat}(\tau_\betac)$ is a compactification of $\Ni_{\boldsymbol\nu}$.
\end{theorem}
\begin{proof}
\label{sigma}
The existence and desingularization property of $\theta$ are implied by the main results of \cite{he} and \cite{fu}.
Recall the injective map $\psi$ from \Cref{springer}; let $\psi_{\boldsymbol\nu}=\psi|\Ni_{\boldsymbol\nu}$.
Note that $\kappa=\tau_\betac\sigma$, where the affine permutation matrix of $\sigma\in W$ is given by \begin{align}
\label{form:sigma}
    \sigma=\sum\limits_{i=1}^{s}\E_{i,l(i)}+\sum\limits_{i=1}^{n-s}\E_{i+s,m(i)}
\end{align}
In particular, $X_{\gOhat}(\tau_\betac)=X_{\gOhat}(\kappa)$, and so it follows from \Cref{injective} that $\psi(\Ni_{\boldsymbol\nu})\subset X_{\gOhat}(\tau_\betac)$.
Since $\dim\Ni_{\boldsymbol\nu}=2\dim\gl/\para=l(\tau_\betac)=\dim X_{\gOhat}(\tau_\betac)$, the result follows.
\end{proof}

\ifdraft\marginpar{nilpotent.tex}\fi
\ifdraft\marginpar{codim.tex}\fi
\section{Dimension of $X_\Para(\kappa)$}
Lakshmibai (cf. \cite{vl}) has constructed, for \para\ a maximal parabolic (proper) subgroup in \gl, a similar embedding $\eta_\para:T^*\gl/\para\hookrightarrow X_\Para(\eta)$. 
The map $\eta_\para$ is dominant, thus giving a Schubert variety as a compactification of $T^*\gl/\para$.
Lakshmibai et al. (cf. \cite{crv}) discuss a family of maps $T^*\gl/\borel\rightarrow \gLhat/\Borel$, including the map $\phi_\borel$, for which the image of $T^*\gl/\borel$ is \emph{not} a Schubert variety.
With this in mind, we compute the dimension of $X_\Para(\kappa)$, and thus the codimension of $\image(\phi_\para)$ in $X_\Para(\kappa)$. 
We also recover Lakshmibai's result, i.e., for \para\ a maximal parabolic (proper) subgroup in \gl, the map $\phi_\para$ is dominant.
In fact, \Cref{maximalresult} says that $X_\Para(\kappa)$ is a compactification of $T^*G/P$ if and only if $P$ is a maximal parabolic subgroup.

\begin{prop}
The dimension of the Schubert variety $X_\Para(\kappa)$ equals the length of $\kappa$.
\end{prop}
\begin{proof}
We need to show that $\kappa\in\What^\Para$. 
For $\alpha_i\in\Simple_\para$, we show $\kappa s_i>\kappa$.
Recall the partitioning of $\left\{1,\ldots,n\right\}$ from \Cref{redb4blue} into \red\ and \blue.
Note that $\alpha_i\in\Simple_\para$, implies $i$ and $i+1$ appear in the same row of the tableau.
In particular, if $i\in\Simple_\para\bigcap\blue(j)$ then $i+1\in\blue(j)$.
\begin{enumerate}
    \item   Suppose $i\in\blue$. 
            Then $i=m(k)$ and $i+1=m(k+1)$ for some $k$. 
            The non-zero entries in the $i^{th}$ and $(i+1)^{th}$ columns are $t^{-1}\E_{k+s,i}$ and $t^{-1}\E_{k+s+1,i+1}$.
            We apply Case 1 of \Cref{table} with $a=i,\,b=i+1$.
    \item   Suppose $i\in\red$ and $i+1\in\blue$.
            The non-zero entries in the $i^{th}$ and $(i+1)^{th}$ columns are $t^{\nu_k-1}\E_{k,i}$ and $t^{-1}\E_{j,i+1}$.
            Since $k\leq s<j$, we can apply Case 2 of \Cref{table} with $a=i,\,b=i+1$, $\nu_k-1=\ord(t_a)>\ord(t_b)=-1$ and $i=\sigma(a)<\sigma(b)=i+1$ to get $\kappa<\kappa s_i$.
    \item   Suppose $i,i+1\in\red$.
            Since $i$ and $i+1$ are in the same row of the tableau, we have $i=l(k)$ and $i+1=l(k+1)$ for some $k$. 
            The non-zero entries in the $i^{th}$ and $(i+1)^{th}$ columns are $t^{\nu_k-1}\E_{k,i}$ and $t^{\nu_{k+1}-1}\E_{k+1,i+1}$.
            If $\nu_k=\nu_{k+1}$, Case 1 of \Cref{table} applies with $a=i,\,b=i+1$ to give $\kappa<\kappa s_i$.
            If $\nu_k>\nu_{k+1}$, Case 2 of \Cref{table} applies with $a=i,\,b=i+1$, $\nu_k-1=\ord(t_a)>\ord(t_b)=\nu_{k+1}-1$ to give $\kappa<\kappa s_i$.
\end{enumerate}
\end{proof}
\begin{lemma}
\label{lsigma}
\label{lengthofkappa}
The length of $\kappa$ is given by the formula \begin{align*}
    l(\kappa)   &=2\dim\gl/\para+\sum\limits_{k'<k}\#\row(k)\#\blue(k')
\end{align*}
\end{lemma}
\begin{proof}
We have $\kappa=\tau_\betac\sigma$ where $\sigma$ is as in \Cref{form:sigma}.
Viewing $\sigma$ as an element of $S_n$, we have $\sigma^{-1}(i)=\begin{cases}l(i)&i\leq s\\m(i-s)&i>s\end{cases}$. 
We compute \begin{align*}
    l(\sigma)=  &\#\left\{(i,j)\mid 1\leq i<j\leq n,\,\sigma^{-1}(i)>\sigma^{-1}(j)\right\} \\
             =  &\#\left\{(i,j)\mid i<j\leq s, l(i)>l(j)\right\}\\ 
                    &\qquad+\#\left\{(i,j)\mid i\leq s<j,\,l(i)>m(j-s)\right\}\\ 
                    &\qquad\qquad+\#\left\{(i,j)\mid s<i<j,\,m(i-s)>m(j-s)\right\} 
\end{align*}
Recall that $l(i)$ is an increasing sequence, i.e., $i<j<s\implies l(i)<l(j)$, and so \begin{align*}
    l(\sigma)=  &\#\left\{(i,j)\mid i\leq s,\,j\leq n-s,\,l(i)>m(j)\right\}     \\ 
                    &\qquad+\#\left\{(i,j)\mid i<j\leq n-s,\,m(i)>m(j)\right\}  \\
             =  &\#\left\{(i,j)\mid i\in\red,\,j\in\blue,\,i>j\right\}     \\ 
                    &\qquad+\#\left\{(i,j)\mid i<j\leq n-s,\,m(i)>m(j)\right\}  \\
             =  &\sum\limits_{k'<k}\#\left\{(i,j)\mid i\in\red(k),\,j\in\blue(k')\right\} \\ 
                    &\qquad+\sum\limits_{k'<k}\#\left\{(i,j)\mid i\in\blue(k),\,j\in\blue(k')\right\} \\
             =  &\sum\limits_{k'<k}\#\left\{(i,j)\mid i\in\row(k),\,j\in\blue(k')\right\} \\ 
             =  &\sum\limits_{k'<k}\#\row(k)\#\blue(k')    
\end{align*}
Finally, \Cref{betainWg} implies $l(\kappa)=l(\tau_\betac)+l(\sigma)$ and we get, as claimed:$$l(\kappa)=2\dim\gl/\para+\sum\limits_{k'<k}\#\row(k)\#\blue(k')$$
\end{proof}

\begin{cor}
\label{maximalresult}
The Schubert variety $X_\Para(\kappa)$ is a compactification of $T^*G/P$ if and only if $P$ is a maximal parabolic subgroup.
\end{cor}
\begin{proof}
The parabolic subgroup $P$ is maximal if and only if $S_\para=\simple\backslash\left\{\alpha_d\right\}$ for some $d$, equivalently, the corresponding tableau has exactly $2$ rows. 
In this case $\blue\subset\row(2)$, (recall from \Cref{redb4blue} that $\blue=\bigsqcup\limits_i\blue(i)$).
It follows that the second term in \Cref{lengthofkappa} is an empty sum, which implies $l(\kappa)=2\dim\gl/\para$.
Suppose now that the tableau has $r\geq 3$ rows.
In this case, both $\blue(2)$ and $\row(r)$ are non-empty, and so the second term in \Cref{lengthofkappa} is strictly greater than $0$.
\end{proof}
\ifdraft\marginpar{codim.tex}\fi
\end{document}